\documentclass%[dvipdfmx,a4paper]
{article}
\usepackage{amsmath, amssymb, amsthm, graphicx,bm,
cite, 
comment}\usepackage{tikz-cd}
\usepackage[all]{xy}
\usepackage{titlesec}\usepackage{calligra,mathrsfs}
\DeclareMathOperator{\gr}{gr}
\DeclareMathOperator{\supp}{supp}

\DeclareMathOperator{\spec}{Spec}
\DeclareMathOperator{\Def}{Def}  
  
\theoremstyle{plain}% default
\newtheorem{thm}{Theorem}[section]
\newtheorem{lem}[thm]{Lemma}
\newtheorem{cor}[thm]{Corollary}
\newtheorem{prop}[thm]{Proposition}

\newtheorem*{thm*}{Theorem}
\newtheorem*{lem*}{Lemma}
\newtheorem*{cor*}{Corollary}

\theoremstyle{definition}
\newtheorem{dfn}[thm]{Definition}

\newtheorem{rmk}[thm]{Remark}

\newtheorem*{dfn*}{Definition}

\newtheorem*{hp*}{Hypothesis}

\numberwithin{equation}{section}\numberwithin{figure}{section}

 \def\reg{{ { \rm reg}  }}

\def\tp#1{{}^t\!#1}

\def\e#1\e{\begin{equation}#1\end{equation}}
\def\iz#1\iz{\begin{itemize}#1\end{itemize}}
\def\ea#1\ea{\e{\begin{split}#1\end{split}}\e}
\def\eq{\eqref}
\def\l{\label}
\def\0{\hspace{0pt}}

\def\ph{\phi}

\def\Uh_#1{\,\widehat{\!U}_{\!#1}}

\def\Ph{\Phi}

\def\ps{\psi}

\def\dim{\mathop{\rm dim}\nolimits}

\def\ker{\mathop{\rm ker}}
\def\im{\mathop{\rm im}}
\def\coker{\mathop{\rm coker}}

\def\Ext{\mathop{\rm Ext}\nolimits}

\def\ge{\geqslant}
\def\le{\leqslant\nobreak}

\def\cE{{{\mathcal E}}}
\def\cF{{{\mathcal F}}}
\def\cG{{{\mathcal G}}}

\def\O{{{\mathcal O}}}

\def\cU{{{\mathcal U}}}

\def\={\equiv}

\def\cX{{{\mathcal X}}}
\def\cY{{{\mathcal Y}}}

\def\C{{{\mathbb C}}}

\def\Q{{{\mathbb Q}}}
\def\R{{{\mathbb R}}}
\def\Z{{{\mathbb Z}}}

\def\al{\alpha}
\def\be{\beta}
\def\ga{\gamma}
\def\de{\delta}
\def\io{\iota}
\def\ep{\epsilon}

\def\ta{{\tau}}
\def\ze{\zeta}

\def\om{\omega}
\def\De{\Delta}

\def\Si{\Sigma}

\def\Th{\Theta}
\def\Om{\Omega}
\def\Ga{\Gamma}

\def\db{{\bar\bd}}%{{\,\ov{\!\partial}}}
\def\ts{\textstyle}%\def\st{\scriptstyle}

\def\-{\setminus}
\def\op{\oplus}

\def\ov{\overline}

\def\ul{\underline}

\def\cm{\circ}

\def\sb{\subseteq}\def\bd{\partial}\def\sing{{\rm sing}}

\begin{document}
%\date{}

\title{Deformations of Compact Calabi--Yau and Fano Varieties with Isolated Singularities}
\date{}
\author{Yohsuke Imagi}
\maketitle

\begin{abstract}
Let $X$ be a compact log-canonical K\"ahler $n$-fold, $n\ge3,$ with trivial canonical sheaf and with isolated singularities. We prove that the generic fibres of the semi-universal deformation of $X$ have Du Bois invariant $b^{1,n-2}=0$ at the singular points. Under a certain topological hypothesis on $X$ the generic fibres have also link invariant $l^{1,n-2}=0.$ If $X$ is a projective log-canonical $n$-fold, $n\ge3,$ with ample anti-canonical sheaf and with isolated singularities then the generic fibres have $b^{1,n-2}=l^{1,n-2}=0$ (without the topological hypothesis). These are generalizations of recent results of Tenie `Global smoothing of singular Fano and Calabi--Yau varieties' and an older result of Namikawa `Deformation theory of Calabi--Yau threefolds and certain invariants of singularities.'   
\end{abstract}

\section{Introduction}
The global smoothing problem for Calabi--Yau and Fano $n$-folds, $n\ge3,$ has been studied since the 1990s \cite{NS, Nam97, Nam Fano}. The following theorems are the fundamental results for Calabi--Yau $3$-folds. 
\begin{thm}[Namikawa--Steenbrink {\cite[Theorems 1.3]{NS}}]\label{thm: NS}
Let $X$ be a complex projective canonical $3$-fold whose canonical sheaf $\om_X$ is trivial and whose singularities are isolated hypersurface singularities. Then there exist an open neighbourhood $\De$ of $0\in\C$ and a proper flat morphism $f:\cX\to\De$ with $f^{-1}(0)\cong X$ and such that the singularities of $f^{-1}(t), t\in \De\-\{0\},$ are ordinary double points. \qed
\end{thm}
\begin{thm}[Namikawa--Steenbrink {\cite[Theorem 2.4]{NS}}]\label{thm: NS2}
Let $X$ be as in Theorem \ref{thm: NS} and suppose that it is $\Q$-factorial. Then there exist an open neighbourhood $\De$ of $0\in\C$ and a proper flat morphism $f:\cX\to\De$ with $f^{-1}(0)\cong X$ and whose other fibres are non-singular. \qed
\end{thm}

Theorem \ref{thm: NS} has recently been generalized as follows.
\begin{thm}[Tenie {\cite[Theorem 1.3]{Ten}}]\label{thm: Ten1}
Let $X$ be a compact log-canonical complex $n$-fold, $n\ge3,$ with $\om_X\cong\O_X$ and $H^1(X,\O_X)=0,$ whose singularities are isolated complete intersection singularities and such that
\begin{equation}\label{ddbar}
\text{if $Y\to X$ is a resolution of singularities then the $\partial\db$ lemma should hold for $Y.$}
\end{equation}
Then there exist an open neighbourhood $\De$ of $0\in\C$ and a proper flat morphism $f:\cX\to S$ with $f^{-1}(0)\cong X$ and such that the singularities of $f^{-1}(t), t\in\De\-\{0\},$ are $1$-Du Bois. The singularities of $X$ which are not $1$-Du Bois become non-singular under this. \qed
\end{thm}
If $X$ is as in Theorem \ref{thm: NS} then by Kawamata \cite[Theorem 8.3]{Kaw1} we have $H^1(X,\O_X)=0$ unless $X$ is non-singular. Also \eq{ddbar} holds because $X$ is projective. On the other hand, by Chen, Dirks and Musta\c{t}\v{a} \cite[(2.9) and Theorem 2.3]{CDM} isolated $1$-Du Bois complete intersection singularities of dimension $n=3$ are hypersurface singularities, which are by Namikawa--Steenbrink \cite[Theorem 2.2]{NS} ordinary double points. Theorem \ref{thm: Ten1} thus generalizes Theorem \ref{thm: NS}. The following theorem is proved in the same paper \cite{Ten}.

\begin{thm}[Tenie {\cite[Theorem 1.4]{Ten}}]\label{thm: Ten2}
Let $X$ be as in Theorem \ref{thm: Ten1}. Suppose that its Du Bois complex $\ul\Om^\bullet_X$ satisfies the equality
\begin{equation}\label{DB complex}
\dim_\C \mathbb{H}^{n-1}(X,\ul\Om^2_X)\le\dim_\C \mathbb{H}^1(X,\ul\Om^{n-2}_X).
\end{equation}
Then there exist an open neighbourhood $\De$ of $0\in\C$ and a proper flat morphism $f:\cX\to\De$ with $f^{-1}(0)\cong X$ and such that the singularities of $f^{-1}(t), t\in\De\-\{0\},$ are $1$-rational. The singularities of $X$ which are not $1$-rational become non-singular under this.
\end{thm}
\begin{rmk}\label{rmk: Ten2}
Since the singularities of $X$ are complete intersections it follows readily that the inequality \eq{DB complex} is equivalent to the equality $\dim_\C \mathbb{H}^{n-1}(X,\ul\Om^2_X)=\dim_\C \mathbb{H}^1(X,\ul\Om^{n-2}_X);$ see Proposition \ref{prop: 45}.
\end{rmk}
If $X$ is as in Theorem \ref{thm: NS2} then its singularities are rational and by Park--Popa \cite[Theorem E]{PP} it is $\Q$-factorial. (More precisely, if $X$ is a projective $3$-fold with rational singularities then \eq{DB complex} with $n=3$ is necessary and sufficient for $X$ to be $\Q$-factorial.) On the other hand, by Steenrbink \cite[Theorem 6]{St97} and Friedman--Laza \cite[Theorems 5.2 and  5.3]{FL iso} isolated $1$-rational complete intersection singularities of dimension $n=3$ have $s_2=b^{11}+l^{11}=0$ which are therefore smooth. Theorem \ref{thm: Ten2} thus generalizes Theorem \ref{thm: NS2}.

In Theorems \ref{thm: NS}--\ref{thm: Ten2} the singularities of $X$ are complete intersection singularities. There is however a known generalization of Theorem \ref{thm: NS2} to the other singularities, which we state now. By a {\it Kuranishi space} we mean the base space of a semi-universal deformation of the germ of a singularity (or a compact complex space).
\begin{thm}[Namikawa {\cite[Theorem 5]{Nam97}}]\label{thm: Nam97}
Let $X$ be a $\Q$-factorical projective canonical $3$-fold with $\om_X\cong\O_X$ and with isolated singularities. Suppose that for every singular point $x\in X$ the germ $(X,x)$ admits a smooth Kuranishi space with at least one smoothing component. Then $X$ is globally smoothable. \qed
\end{thm}

In the present paper we generalize the principal part of Tenie's theorems, Theorems \ref{thm: Ten1} and \ref{thm: Ten2}, using the method of proof of Theorem \ref{thm: Nam97}. The following is the generalization of Theorem \ref{thm: Ten1}.
\begin{thm}\label{thm: CY1}
Let $X$ be a compact K\"ahler log-canonical $n$-fold, $n\ge3,$ with $\om_X\cong\O_X$ and with isolated singularities. Then there exist a semi-universal deformation $f:\cX\to\Def(X)$ and a Zariski dense open subset $\De\sb \Def(X)$ such that for every $t\in\De$ the singularities of $f^{-1}(t)$ have $b^{1,n-2}=0.$ 
\end{thm}

Being K\"ahler is stronger than \eq{ddbar}. It is however, as we prove in Corollary \ref{cor: Kahler}, an open condition under the deformations of $X.$ This is indispensable to us because we follow the proof of Theorem \ref{thm: Nam97}. In Theorem \ref{thm: Nam97}, unless $X$ is non-singular we have $H^1(X,\O_X)=0$ and so by Serre duality $H^2(X,\O_X)=0.$ Projective is therefore an open condition. In Theorem \ref{thm: CY1} K\"ahler plays the same r\^ole.

The condition $b^{1,n-2}=0$ is a weaker version of the $1$-Du Bois condition. They are equivalent for complete intersection singularities \cite[Theorem 5.2]{FL iso}. In Theorem \ref{thm: CY1} the singularities need not be even locally smoothable and it is too strong to claim that the singularities with $b^{1,n-2}\ne0$ disappear as in Theorem \ref{thm: Ten1}. In Theorem \ref{thm: CY1} they become singularities with $b^{1,n-2}=0.$

To prove Theorem \ref{thm: CY1} we show first that there exists a Zariski dense open subset $D\sb\Def(X)$ over which the deformation family admits a simultaneous resolution. Then $b^{1,n-2}$ is an upper semi-continuous function on $D.$ The heart of Theorem \ref{thm: CY1} is that the set of $b^{1,n-2}=0$ is dense.  

The generalization of Theorem \ref{thm: Ten2} is as follows. 
\begin{thm}\label{thm: CY2}
Let $X$ be as in Theorem \ref{thm: CY1}. Denote by $X^\reg\sb X$ the regular locus and suppose that
\begin{equation}\label{top}
\text{the connecting homomorphism $H^n(X^\reg,\C)\to H^{n+1}(X,X^\reg;\C)$ is surjective.}
\end{equation}
Then there exist a semi-universal deformation $f:\cX\to\Def(X)$ and a Zariski dense open subset $\De\sb \Def(X)$ such that for every $t\in\De$ the singularities of $f^{-1}(t)$ have $b^{1,n-2}=l^{1,n-2}=0.$
\end{thm}

The hypothesis \eq{top} is slightly stronger than \eq{DB complex}. More precisely, there is an exact sequence 
\[
\gr^{n-1}_FH^n(X^\reg,\C)\to \gr^{n-1}_FH^{n+1}(X,X^\reg;\C)\to \mathbb{H}^{n-1}(X,\ul\Om^2_X)\to \mathbb{H}^1(X,\ul\Om^{n-2}_X).
\]
As $H^n(X^\reg,\C)\to H^{n+1}(X,X^\reg;\C)$ is a morphism of mixed Hodge structures, \eq{top} implies \eq{DB complex}. For isolated rational singularities of dimension $n=3$ the relation is clearer:
\begin{thm}\label{4}
Let $X$ be a compact complex $3$-fold with isolated rational singularities and whose resolutions satisfy the $\partial\db$ lemma. Then for $n=3$ the conditions \eq{DB complex} and \eq{top} are equivalent. Moreover, if $X$ is an algebraic variety then they are necessary and sufficient for $X$ to be $\Q$-factorial.
\end{thm}
As \eq{top} concerns only the (singular) cohomology groups, it is an open condition under the deformations of $X.$ This is again indispensable because we follow the proof of Theorem \ref{thm: Nam97}. (In Theorem \ref{thm: Nam97} $\Q$-factorial is an open condition, by Koll\'ar--Mori \cite[Theorem 12.1.10]{MK}.)
 
The condition $b^{1,n-2}=l^{1,n-2}=0$ is a weaker version of the $1$-rational condition. They are equivalent for complete intersection singularities \cite[Theorems 5.2 and 5.3]{FL iso}. In Theorem \ref{thm: CY2} the singularities need not be even locally smoothable and it is too strong to claim that the singularities with $b^{1,n-2}+l^{1,n-2}\ne0$ disappear as in Theorem \ref{thm: Ten2}. In Theorem \ref{thm: CY2} they become singularities with $b^{1,n-2}=l^{1,n-2}=0.$

For $n=3,$ by Namikawa--Steenbrink \cite[Theorem 1.1]{NS} the singularities with $b^{11}=l^{11}=0$ are rigid. Combining this with Theorem \ref{thm: CY2} we prove 
\begin{cor}\label{cor: CY2}
Let $X$ be a compact K\"ahler log-canonical $3$-fold with $\om_X\cong\O_X,$ with isolated singularities and satisfying \eq{top} with $n=3.$ Suppose that for every singular point $x\in X$ 
\begin{equation}\label{strong sm}\parbox{10cm}{
there exist a semi-universal deformation $g:\cU\to\Def(X,x)$ and a Zariski dense open subset $S\sb\Def(X,x)$ such that for every $s\in S$ the fibre $g^{-1}(s)$ is non-singular.  
}\end{equation}
Then there exist a semi-universal deformation $f:\cX\to\Def(X)$ and a Zariski dense open subset $\De\sb \Def(X)$ such that for every $t\in\De$ the fibre $f^{-1}(t)$ is non-singular.
\end{cor}
\begin{proof}
Let $f:\cX\to\Def(X)$ and $\De\sb\Def(X)$ be as in Theorem \ref{thm: CY2}. By \eq{strong sm}, after making $\Def(X)$ smaller if necessary, the singularities of $f^{-1}(t),t\in\De,$ are locally smoothable. But as $b^{11}=l^{11}=0$ they are also rigid. This means that $f^{-1}(t)$ is non-singular. 
\end{proof}

Corollary \ref{cor: CY2} generalizes Theorem \ref{thm: Nam97}. The hypothesis \eq{strong sm} is weaker than that of Theorem \ref{thm: Nam97}.
  
We turn now to the deformations of Fano varieties. We generalize the principal part of the following theorem.
\begin{thm}[Tenie {\cite[Theorem 1.5]{Ten}}]\label{thm: Ten Fano}
Let $X$ be a projective log-canonical Gorenstein $n$-fold, $n\ge3,$ with isolated complete intersection singularities and such that $\om_X^{-1}$ is ample. Then there exist an open neighbourhood $\De$ of $0\in\C$ and a proper flat morphism $f:\cX\to\De$ with $f^{-1}(0)\cong X$ and such that the singularities of $f^{-1}(t), t\in\De\-\{0\},$ have $b^{1,n-2}=l^{1,n-2}=0.$ The singularities of $X$ with $b^{1,n-2}+l^{1,n-2}\ne0$ become non-singular under this.
\end{thm}

The following theorem is our main result. 
\begin{thm}\label{thm: Fano}
Let $X$ be a projective log-canonical Gorenstein $n$-fold, $n\ge3,$ with isolated singularities and such that $\om_X^{-1}$ is ample. Then there exist a semi-universal deformation $f:\cX\to\Def(X)$ and a Zariski dense open subset $\De\sb \Def(X)$ such that for every $t\in\De$ the singularities of $f^{-1}(t)$ have $b^{1,n-2}=l^{1,n-2}=0.$
\end{thm}
Here is no counterpart to \eq{top} and the statement is simpler than that of Theorem \ref{thm: CY2}. Again the singularities of $X$ need not even be locally smoothable and it is too strong to claim that the singularities with $b^{1,n-2}+l^{1,n-2}\ne0$ disappear as in Theorem \ref{thm: Ten Fano}. In Theorem \ref{thm: Fano} they become singularities with $b^{1,n-2}=l^{1,n-2}=0.$ For $n=3,$ combining Theorem \ref{thm: Fano} with Namikawa--Steenbrink's rigidity theorem \cite[Theorem 1.1]{NS} we obtain an analogue of Corollary \ref{cor: CY2}:
\begin{cor}\label{cor: Fano}
Let $n=3$ and let $X$ be as in Theorem \ref{thm: Fano}. Suppose that every singular point $x\in X$ satisfies \eq{strong sm}. Then there exist a semi-universal deformation $\cX\to\Def(X)$ and a Zariski dense open subset $S\sb \Def(X)$ such that for every $s\in S$ the fibre $\cX_s$ is non-singular. \qed
\end{cor}

The feature of our results including the Calabi--Yau case is that the Kuranishi space $\Def(X)$ is possibly singular. In Theorems \ref{thm: NS}--\ref{thm: Ten2}, \ref{thm: Nam97} and \ref{thm: Ten Fano} it is non-singular (by the results of \cite{Nam94, Gross, Friedman}) and moreover the singularities are complete intersections. The proof of Theorem \ref{thm: Ten1}, for instance, is as follows. Let $x\in X$ be a singular point with $b^{1,n-2}(X,x)\ne0.$ Then there is a corresponding tangent vector $\Def(X,x)$ which has the effect of smoothing the germ $(X,x).$ Varying $x$ we get finitely many tangent vectors, and there is no obstruction to extending them to a global deformation. In the proof of Theorem \ref{thm: Ten2} the singular points of $X$ with $b^{1,n-2}+l^{1,n-2}\ne0$ have local smoothings. There is now an obstruction to extending them to a global deformation, which is exactly the hypothesis \eq{DB complex}.

For singularities which are not complete intersections we use the method of proof of Theorem \ref{thm: Nam97}. For each singular point $x\in X$ there is a stratification of the Kuranishi space $\Def(X,x)$ such that over each stratum there is a simultaneous resolution of the deformation family. The origins of the Kuranishi spaces $\Def(X)$ and $\Def(X,x)$'s are all denoted by $o.$ There exists a tangent vector $v\in T_o\Def(X)$ such that if $b^{11}(X,x)+l^{11}(X,x)\ne0$ then the image of $v$ in $T_o\Def(X,x)$ is not tangent to the stratum containing $o.$ This implies that there exists a global deformation of $X$ which has the effect of moving to the higher-dimensional strata of $\Def(X,x)$'s. As there are finitely many strata, the induction terminates. 

This in fact does not require $\Def(X)$ to be non-singular. Suppose that the deformation family over $\Def(X)$ has a simultaneous resolution and denote by $(Y,E)\to(X,X^\sing)$ the induced resolution of $X.$ By Koll\'ar--Mori \cite[Proposition 11.4]{MK} there exist certain morphisms $\Def(Y,E)\to\Def(Y)\to\Def(X)$ and we show that the induced map $T_o\Def(Y,E)\to T_o\Def(X)$ is surjective. (Roughly speaking, the simultaneous contraction of the deformation family over $\Def(Y,E)$ or $\Def(Y,E)$ is the same as the original family. For the details see Theorem \ref{thm: surj tang}.) The same computation of tangent spaces as in the proof of Theorem \ref{thm: Nam97} shows that the singularities of $X$ have $b^{11}=l^{11}=0.$ In other words, if $X$ has singular points with $b^{11}+l^{11}\ne0$ then no simultaneous resolutions exist over $\Def(X).$ This means that we can move to the higher-dimensional strata of $\Def(X,x)$'s.   

We thus generalize Theorem \ref{thm: Nam97} to Corollary \ref{cor: CY2}. In the same way we generalize Theorems \ref{thm: Ten1}, \ref{thm: Ten2} and \ref{thm: Ten Fano} to Theorems \ref{thm: CY1}, \ref{thm: CY2} and \ref{thm: Fano} respectively. (As we have seen above, Corollary \ref{cor: CY2} follows immediately from Theorem \ref{thm: CY2}. In the later sections therefore only Theorem \ref{thm: CY2} is proved.)

We begin in \S\ref{sect: SR} and \S\ref{sect: SC} with the treatment of simultaneous resolutions and contractions. In \S\ref{sect: SR} we introduce certain stratifications of $\Def(X,x)$'s following the proof of Theorem \ref{thm: Nam97}. In \S\ref{sect: SC} we prove that $T_o\Def(Y,E)\to T_o\Def(X)$ is surjective as stated above.

In \S\ref{sect: Kahl} we prove that K\"ahler is an open condition with respect to the deformation parameters. For this we use the result of Bingener \cite{Bing} and the method of proof of Namikawa \cite[Proposition 5]{Nam2001}.

In \S\ref{sect: lc}--\S\ref{sect: vanishing} we recall the relevant facts about log-canonical and Du Bois singularities. By Ishii {\cite[Theorem 2.3]{Ish G}} log-canonical and Gorenstein are equivalent for isolated Gorenstein singularities. In Theorem \ref{thm: lc} we prove that log-canonical is an open condition under the deformations of $X.$ This is well known if the base space is a curve. In our main results however the base space is the whole Kuranishi space $\Def(X)$ and we therefore prove the openness statement.

In \S\ref{sect: DB} we recall the basis properties of isolated Du Bois singularities. In \S\ref{sect: vanishing} we review the relevant vanishing theorems for isolated Du Bois singularities. Theorem \ref{thm: PSV} is a result of Popa, Shen and Vo \cite[Proposition 5.2]{PSV} (or Friedman \cite[Remark 2.5]{Friedman}) which is the key to Tenie's proof of Theorems \ref{thm: Ten1} and \ref{thm: Ten2}. For this we do not need the singularities to be complete intersections and we can therefore use it to prove Theorems \ref{thm: CY1} and \ref{thm: CY2}. 

In \S\ref{sect: proof of CY1} we prove Theorem \ref{thm: CY1}. This is similar to the proof of Theorem \ref{thm: Ten1}, apart from the difference explained above.  

In \S\ref{sect: proof of CY2} we prove Theorem \ref{thm: CY2}. This is modelled on the proof of Theorem \ref{thm: Nam97} but part of it is special to rational singularities of dimension $n=3.$ The latter is replaced by a certain computation using the techniques of Friedman--Laza \cite[Theorem 2.1(iii),(v)]{FL} which apply to Du Bois singularities of dimension $n\ge3.$ 

In \S\ref{sect: proof of 4} we prove Theorem \ref{4}. This is essentially the same as the proof of Namikawa--Steenbrink \cite[Lemma 3.5]{NS}.

In \S\ref{sect: proof of Fano} we prove Theorem \ref{thm: Fano}. This is similar to the proof of Theorem \ref{thm: Ten Fano}, apart from the difference explained above.

\section{Simultaneous Resolutions}\label{sect: SR}
A {\it simultaneous resolution of complex spaces} is a resolution of a complex space morphism, which we define as follows.
\begin{dfn}\label{dfn: simul resol}
Let $X,S$ be complex spaces and $X\to S$ a flat morphism with reduced fibres. A {\it resolution of the morphism $X\to S$} consists of a complex space $Y,$ a smooth morphism $Y\to S$ and a proper bimeromorphic $S$-morphism $f:Y\to X$ which induces fibrewise resolutions of singularities. 
\end{dfn}

We define a relative version of simple normal crossing divisors.
\begin{dfn}\label{dfn: relative snc}
Let $Y,S$ be complex spaces and $f:Y\to S$ a smooth morphism. A {\it simple normal crossing divisor on $Y\to S$} is an analytic subset $D\subset Y$ such that for every $y\in D$ there exist an open neighbourhood $U\sb Y$ of $y,$ an open neighbourhood $V\sb S$ of $f(y),$ an open neighbourhood $W\sb f^{-1}(f(y))$ of $y,$
an $S$-isomorphism $U\to V\times W$ and holomorphic local coordinates $z_1,\dots,z_l:W\to\C$ such that $U\cap D$ is defined by $z_1\cdots z_k=0$ for some $k\le l.$ 
\end{dfn}

Combining Definitions \ref{dfn: simul resol} and \ref{dfn: relative snc} we define good resolutions of morphisms. 
\begin{dfn}
Let $X,S$ be complex spaces and $X\to S$ a flat morphism with reduced fibres. The singular set of the fibre $X_s, s\in S,$ is denoted by $X_s^\sing.$ A {\it good resolution of $X\to S$} is a resolution $f:Y\to X$ such that $\coprod_{s\in S}f^{-1}(X_s^\sing)^{\rm red}$ is a simple normal crossing divisor on the composite morphism $Y\to S.$
\end{dfn}

We recall next the definition of critical points and critical values for morphisms of non-singular complex spaces.
\begin{dfn}
Let $X,Y$ be non-singular complex spaces and $f:X\to Y$ a morphism. We call a point $x\in X$ a {\it critical point of $f$} if the tangent map $T_xf:T_xX\to T_{f(x)}Y$ is not surjective. We call a point $y\in Y$ a {\it critical value of $f$} if $y=f(x)$ for some critical point $x\in X.$ 
\end{dfn}
Combining Sard's theorem and the proper mapping theorem, we obtain 
\begin{lem}\l{lem: Sard}
Let $X,Y$ be non-singular complex spaces and $f:X\to Y$ a morphism. Let $K\sb X$ be a closed subset which contains all the critical points of $f$ and such that $f|_K:K\to Y$ is a proper map. Then the set of critical values of $f$ is an analytic subset of $Y$ of positive codimension. \qed
\end{lem}

We show now that if $(X,x)\to (S,o)$ is a deformation of an isolated singularity $x$ of a reduced complex space $X_o$ then $S$ has a certain stratification corresponding to simultaneous resolutions. More precisely:
\begin{thm}\label{thm: strat}
Let $X,S$ be complex spaces and $X\to S$ a flat morphism of complex spaces with reduced fibres. Let $o\in S$ be a point such that the fibre $X_o$ over it has an isolated singularity $x.$ Then there exist, after making $X$ smaller if necessary, finitely many disjoint subsets $S_0,\dots,S_m\sb S$ with $S_0\sqcup\dots\sqcup S_m=S$ and such that the following hold. 

For every $i=0,\dots,m$ the stratum $S_i\sb S$ is a Zariski locally closed, non-singular and of pure dimension. We have $\dim S_0<\dots<\dim S_m;$ the first stratum $S_0$ contains the origin $o\in S$ and the last stratum $S_m\sb S$ is a Zariski open subset. For $i>0$ the closure of $S_i$ is equal to $S_0\sqcup\dots\sqcup S_{i-1}.$ For every $i$ there exists a good resolution of the projection $S_i\times_S X\to S_i.$ 
\end{thm}
\begin{proof}
We modify slightly the proof of Namikawa--Steenbrink \cite[Theorem 2.4]{NS}. The difference is that we allow $S$ to be singular and require the resolutions to be good. 

Let $S^{\rm red}\to S$ be the reduction morphism and $R\to S^{\rm red}$ a resolution of $S^{\rm red}.$ Put $Y:=R\times_S X.$ Since $X_o$ is reduced and $X\to S$ flat it follows, making $S$ smaller if necessary, that $Y$ is reduced. Let $Z\to Y$ be a good resolution of $Y.$ Make $X$ so small that the exceptional divisor of $Z\to Y$ consists of finitely many irreducible components $D_0,\dots,D_n\to R.$ The projection $Y\to R$ induces a smooth morphism $Y^\reg\to R$ and so the composite morphism $Z\-(D_0\cup\dots\cup D_n)\to R$ is a smooth morphism too. Since $x\in X_o$ is an isolated singularity it follows, making $X$ smaller if necessary, that the projection $Y^\sing\to R$ is a finite morphism. It is in particular a proper map and accordingly so is the composite map $D_0\cup\dots\cup D_n\to Y^\sing\to R.$ Applying Lemma \ref{lem: Sard} to the morphism $Z\to R$ with $K=D_0\cup\dots\cup D_n$ we see that its critical value set, denoted by $A\subset R,$ is an analytic subset of positive codimension. In the same way, for each $i=0,\dots,n$ the critical value set of the morphism $D_i\subset Z\to R,$ denote by $B_i\subset R,$ is an analytic subset of positive codimension. The exceptional set of the resolution $R\to S^{\rm red},$ denoted by $E\subset R,$ is also an analytic subset of positive codimension. Set $S^0:=R\-(E\cup A\cup B_0\cup\dots\cup B_n)$ which is a non-empty open subset. Identify $S^0\sb R\-E$ with its image under the composite map $R\to S^{\rm red}\to S.$ Then the projection $S^0\times_S X\to S^0$ is well defined and the morphism $Z\to S$ induces a good resolution of it.

Suppose now that we have defined $S^0,\dots,S^{i-1}$ for $i\ge1.$ Applying the process above with $S\-(S^0\sqcup\dots\sqcup S^{i-1})$ in place of $S$ we define the Zariski open subset $S^i\sb S\-(S^0\sqcup\dots\sqcup S^{i-1}).$ Then $\dim S^i$ is a strictly decreasing function of $i$ and the induction terminates in finitely many steps. The results are finitely many subsets $S^0,\dots,S^m.$ Set $S_0:=S^m,\dots,S_m:=S^0.$ Then $S_0,\dots,S_m$ have the required properties.
\end{proof}

\section{Simultaneous Contractions}\label{sect: SC}
For a complex space $Y$ a {\it contraction} of it consists of a complex space $X$ and a morphism $f:Y\to X$ with $f_*\O_Y=\O_X.$ In this section we prove existence and uniqueness results for simultaneous contractions. (The precise meanings of simultaneous contractions will depend on the contexts. We shall therefore state our results directly without making a universal definition of them.) We begin by recalling an existence result proved in \cite[Proposition 11.4]{MK}.
\begin{prop}\label{prop: MK}
Let $X,Y$ be compact complex spaces and $f:Y\to X$ a morphism with $f_*\O_Y=\O_X$ and $R^1f_*\O_X=0.$ Let $(S,o),(T,o)$ be complex space germs, $\cX\to S$ a versal deformation of $X,$ and $\cY\to T$ a deformation of $Y.$ Then there exist, after making $T$ smaller if necessary, a complex space morphism $\cY\to \cX$ and a germ morphism $(T,o)\to (S,o)$ such that the diagram
\begin{equation}\label{MK}
\begin{tikzcd} Y\ar[r]\ar[d,"f"]& \cY\ar[r]\ar[d] & T\ar[d]\\  X\ar[r]& \cX\ar[r]& S\end{tikzcd}
\end{equation}
commutes.
\end{prop}
\begin{proof}
We recall the proof because its method will be important. Suppose first that $T=\spec A$ for some Artin local $\C$-algebra $A.$ Writing $A$ as successive small extensions of $\C$ we see that $R^1f_*\O_\cY=0$ and ${\rm Tor}_1^A(f_*\O_\cY,\C)=0.$ The latter implies that $f_*\O_Y$ is flat over $A$ and so there is a deformation $\cX'\to T$ of $X$ defined by $\O_{\cX'}=f_*\O_\cY.$ Since $\cX\to S$ is versal we get a morphism $(T,o)\to (S,o)$ such that $\cX'\cong T\times_S\cX.$ Using the projection $T\times_S\cX\to\cX$ define the composite morphism $\cY\to\cX'\cong T\times_S\cX\to\cX.$ Then \eq{MK} commutes.

If $T$ is a general complex space then there exist, as we have seen above, formal morphisms $\cY\to\cX$ and $(T,o)\to (S,o)$ such that \eq{MK} commutes. Denote by $U\to T\times S$ the base  of the relative Douady space of $\cY\times\cX\to T\times S.$ As $\cY\to S$ is proper, there exists an open subspace of $U$ which parametrizes $\cY\to\cX$ and $(T,o)\to(S,o)$ such that \eq{MK} commutes. Using the projection $T\times S\to T$ define the composite morphism $U\to T\times S\to T.$ Then there is a formal section of $U\to T$ corresponding to the formal morphisms $\cY\to\cX$ and $(T,o)\to (S,o).$ Applying the Artin approximation theorem to it and making $T$ smaller if necessary, we get morphisms $\cY\to\cX$ and $(T,o)\to (S,o)$ such that \eq{MK} commutes.
\end{proof}
\begin{dfn}\label{dfn: MK}
Let $X$ be a compact Gorenstein normal complex space of dimension $\ge3,$ and $Y\to X$ a resolution. Let $\Def(X),\Def(Y)$ be the parameter spaces of semi-universal deformations of $X,Y$ respectively. By the hypotheses on $X$ we have $f_*\O_Y=\O_X$ and $R^1f_*\O_X=0.$ We can therefore apply Proposition \ref{prop: MK} to the semi-universal deformations of $X,Y.$ As a result we get a morphism $\Def(Y)\to\Def(X).$
\end{dfn}

There is a uniqueness result for simultaneous contractions of schemes, proved in \cite[Lemma 1.2]{Wah1976}. We make a complex space version of it.
\begin{lem}\label{lem: formal isom}
Let $X,Y$ be complex spaces and $f:Y\to X$ a morphism with $f_*\O_Y=\O_X.$ Let $A$ be an Artin local $\C$-algebra, $\cX\to\spec A$ a deformation of $X,$ and $\cY\to\spec A$ a deformation of $Y.$ Let $\cY\to\cX$ be a $\spec A$ morphism such that the diagram $\begin{tikzcd}
\cY\ar[d]& Y\ar[d,"f"]\ar[l]\\
\cX& X\ar[l]
\end{tikzcd}$
commutes. Then the sheaf morphism $\O_\cX\to f_*\O_\cY$ associated with $\cY\to \cX$ is an isomorphism. 
\end{lem}
\begin{proof}
We proceed by an induction on the $\C$-vector space dimension of $A.$ Let $B$ be an Artin local $\C$-algebra and $A\to B$ be a small extension homomorphism whose kernel is a non-zero principal ideal $(\ep)\sb A.$ Put $\cX':=\spec B\times_{\spec A}\cX$ and $ \cY':=\spec B\times_{\spec A}\cY.$ Then there are commutative diagrams
\[
\begin{tikzcd}\cY\ar[r]& \cX\ar[r]&\spec A\\ \cY'\ar[r]\ar[u]& \cX'\ar[r]\ar[u]&\spec B\ar[u]\end{tikzcd}\text{ and }
\begin{tikzcd}
B\otimes_Af_*\O_\cY\ar[r]&f_*\O_{ \cY'}\\
B\otimes_A\O_\cX\ar[r,"\sim"]\ar[u]& \O_{ \cX'}\ar[u].
\end{tikzcd}
\]
By the induction hypothesis the sheaf morphism $\O_\cX'\to f_*\O_\cY'$ is an isomorphism. In the right-hand commutative diagram therefore the top horizontal arrow $B\otimes_Af_*\O_\cY\to \O_X$ is surjective. So the composite morphism $f_*\O_\cY\to B\otimes_Af_*\O_\cY\to  f_*\O_{ \cY'}$ is surjective too. 

On the other hand, the $A$-module exact sequence $0\to(\ep)\to A\to B\to0$ induces a sheaf exact sequence $0\to(\ep)\otimes_\C\O_Y\to  \O_\cY\to \O_{ \cY'}\to 0;$ and as $f_*$ is left exact, there is an exact sequence $0\to(\ep)\otimes_\C f_*\O_Y\to  f_*\O_\cY\to f_*\O_{ \cY'}.$ But also, as we have seen above, the morphism $f_*\O_\cY\to f_*\O_{ \cY'}$ is surjective. So there is a commutative diagram 
\[\begin{tikzcd}
0\ar[r] &(\ep)\otimes_\C\O_X\ar[r]\ar[d,"\cong"]&  \O_\cX\ar[r]\ar[d]& \O_{ \cX'}\ar[r]\ar[d,"\cong"]& 0\\
0\ar[r] &(\ep)\otimes_\C f_*\O_Y\ar[r]&  f_*\O_\cY\ar[r]& f_*\O_{ \cY'}\ar[r]& 0
\end{tikzcd}\] 
with exact rows. By the five lemma the middle vertical arrow is an isomorphism too, as we have to prove. 
\end{proof}

Lemma \ref{lem: formal isom} is about the deformations over $\spec A.$ For deformations over general complex spaces, we prove
\begin{thm}\label{thm: isom def}
Let $X,Y$ be compact complex spaces and $f:Y\to X$ a morphism with $f_*\O_Y=\O_X.$ Let $\cX\to S$ be a deformation of $X,$ $\cY\to S$ a deformation of $X,$ and $F:\cY\to\cX$ an $S$-morphism such that the diagram $\begin{tikzcd} \cY\ar[d,"F"] & Y\ar[l]\ar[d,"f"] \\ \cX& X\ar[l] \end{tikzcd}$ commutes. Let $\cX'\to S$ be another deformation of $X,$ and $F':\cY\to \cX'$ an $S$-morphism such that that the diagram $\begin{tikzcd} Y\ar[r]\ar[d,"f"]& \cY\ar[d,"F' "] \\ X\ar[r]&  \cX' \end{tikzcd}$ commutes. Then after making $S$ smaller if necessary, the deformations $\cX\to S$ and $\cX'\to S$ are isomorphic.
\end{thm}
\begin{proof}
Identify $X$ with the fibres of $\cX,\cX'$ over the origin of $S.$ By Lemma \ref{lem: formal isom} there exists a formal $S$-isomorphism $\Ph:\cX\to\cX'$ which induces the identity morphism upon $X.$ We proceed now as in the second paragraph of the proof of Proposition \ref{prop: MK}. Denote by $U\to S$ the base of the relative Douady space of $\cX\times_S\cX'\to S.$ As $\cX\to S$ is proper, there exists an open subspace of $U$ which parametrizes $S$-morphisms from $\cX$ to $\cX'.$ The formal $S$-morphism $\Ph:\cX\to\cX'$ defines then a formal section of $U\to S.$ Applying the Artin approximation theorem to it and making $S$ smaller if necessary, we get an $S$-morphism $P:\cX\to\cX'$ which induces the identity morphism $X\to X$ and such that for every $x\in X$ the tangent map $T_x P:T_x\cX\to T_x\cX'$ is equal to the tangent map $T_x\Ph:T_x\cX\to T_x\cX'.$ 

Applying this process to the inverse morphism $\Ph^{-1}: \cX'\to\cX$ we get an $S$-morphism $Q:\cX'\to\cX$ which induces the identity morphism upon $X$ and such that for every $x\in X$ the tangent map $T_xQ:T_x\cX'\to T_x\cX$ is equal to the tangent map $T_x\Ph^{-1}:T_x\cX'\to T_x\cX.$ Now $Q\cm P:\cX\to \cX$ is an $S$-morphism which induces the identity morphism upon $X$ and such that for every $x\in X$ the tangent map $T_x( Q\cm P):T_x\cX\to T_x\cX$ is the identity map. Using the Nakayama lemma twice, we see that the stalk homomorphism $(Q\cm P)_x: \O_{\cX,x}\to \O_{\cX,x}$ is surjective. As $\O_{\cX,x}$ is Noetherian, it is injective too; that is, $(Q\cm P)_x$ is an isomorphism for every $x\in X.$ After making $S$ smaller if necessary, therefore, $Q\cm P:\cX\to\cX$ is an $S$-isomorphism. Now $(Q\cm P)^{-1}\cm Q$ is a left inverse to $P.$

In the same way, making $S$ smaller if necessary, it follows that $P\cm  Q:\cX'\to\cX'$ is an isomorphism and that $Q\cm( P\cm Q)^{-1}$ is a right inverse to $P.$ Thus $P$ is an isomorphism, which completes the proof.  
\end{proof}

We now state and prove the main result of this section.
\begin{dfn}\label{dfn: surj tang}
Let $X$ be a compact Gorenstein normal complex space of dimension $\ge3.$ Let $\cX\to\Def(X)$ be a semi-universal deformation of $X$ and suppose that it has a good resolution $(\cY,\cE)\to(\cX,\cX^\sing).$ The induced resolution of $X$ is denoted by $f:(Y,E)\to(X,X^\sing).$ Since $X$ is Cohen--Macaulay it follows that $R^1f_*\O_Y=0.$

Let $\Def(Y,E)$ be a Kuranishi space of locally trivial deformations of $(Y,E),$ and $\Def(Y)$ a Kuranishi space of deformations of $Y.$ Note that there is an inclusion $\Def(Y,E)\to\Def(Y)$ and let $\Def(Y)\to\Def(X)$ be as in Definition \ref{dfn: MK}. Composing them we get a morphism $\ph:\Def(Y,E)\to\Def(X)$ which induces a map $T_o\ph:T_o\Def(Y,E)\to T_o\Def(X)$ between the tangent spaces. 
\end{dfn}
\begin{thm}\label{thm: surj tang}
In the circumstances of Definition \ref{dfn: surj tang} the map $T_o\ph$ is surjective.
\end{thm}
\begin{proof}
The composite morphism $\cY\to\cX\to\Def(X)$ defines a deformation of $(Y,E)$ and there is a corresponding morphism $\Def(X)\to\Def(Y,E)$ which we denote by $\ps.$ Denote by $\cY'\to\Def(Y,E)$ the morphism which defines the semi-universal deformation of $(Y,E)$ and by $\cY''\to\Def(Y)$ the morphism which defines the semi-universal deformation of $Y.$ Then there is a commutative diagram
\begin{equation}\label{Y}
\begin{tikzcd}  \cY\ar[r]\ar[d]& \cY'\ar[r]\ar[d]& \cY''\ar[r]\ar[d]& \cX\ar[d] \\ \Def(X)\ar[r,"\ps",swap]&\Def(Y,E)\ar[rr, bend right, "\ph"]\ar[r]&\Def(Y)\ar[r]&  \Def(X).\end{tikzcd}
\end{equation}
Denote by $(\ps\cm\ph)^*\cX\to\Def(X)$ the fibre product of $\ps\cm\ph:\Def(X)\to\Def(X)$ and $\cX\to\Def(X).$ Then \eq{Y} induces a $\Def(X)$ morphism $\cY\to(\ps\cm\ph)^*\cX.$ Since \eq{Y} is compatible with the resolution $Y\to X$ it follows that the diagram \begin{tikzcd}Y\ar[r]\ar[d] &X\ar[d]\\ \cY\ar[r]&(\ps\cm\ph)^*\cX\end{tikzcd} commutes. Also the resolution $\cY\to\cX$ of the morphism $\cX\to\Def(X)$ fits into the commutative diagram \begin{tikzcd}Y\ar[r]\ar[d] &X\ar[d]\\ \cY\ar[r]&\cX.\end{tikzcd} Applying Theorem \ref{thm: isom def} to the two diagrams, we get a $\Def(X)$ isomorphism $(\ph\cm\ps)^*\cX\cong\cX.$ Since $\cX\to \Def(X)$ is semi-universal it follows that $T_o(\ps\cm\ph):T_o\Def(X)\to T_o\Def(X)$ is the identity map. Thus $T_o\ph$ is surjective. 
\end{proof}

\section{Deformations of K\"ahler Spaces}\label{sect: Kahl}

We prove now that K\"ahler is an open condition under deformations of complex complex spaces. We use the following theorem.
\begin{thm}[Bingener \cite{Bing}]\label{thm: Bing}
Let $(X,\O_X)$ be a compact K\"ahler space and $\R_X\sb\O_X$ the constant sheaf with stalk $\R$ on $X.$ Suppose that the induced map $H^2(X,\R)\to H^2(X,\O_X)$ is surjective. Then for every deformation $\cX\to\De$ of $X$ there exists a neighbourhood $S$ of $0\in\De$ such that over every $s\in S$ the fibre $\cX_s\sb \cX$ is K\"ahler. \qed
\end{thm}
Here is the key lemma. 
\begin{lem}\label{lem: Kahler}
Let $(X,\O_X)$ be a compact normal complex space with isolated singularities and whose resolutions satisfy the $\partial\db$ lemma. Then the map $H^2(X,\R)\to H^2(X,\O_X)$ induced by $\R_X\subset\O_X$ is surjective.
\end{lem}
\begin{proof}
We modify the proof of \cite[Proposition 5]{Nam2001}. Let $\pi:(Y,E)\to(X,X^\sing)$ be a projective good resolution. The complex $\ul\Om^\bullet_E=\Omega^\bullet_E/\ta^\bullet_E$ is a resolution of the constant sheaf $\C_E.$ Let $V\sb Y$ be an open neighbourhood of $E$ homotopy equivalent to $E.$ Its ordinary de Rham complex $\Om^\bullet_V$ is a resolution of $\C_V.$ Likewise, $\Om^\bullet_Y$ is a resolution of $\C_Y.$ There are commutative diagrams
\begin{equation}\label{YVE}
\begin{tikzcd}[column sep=small]
\C_Y\ar[r]\ar[d]& \C_V\ar[r]\ar[d]& \C_E\ar[d] \\
\Om^\bullet_Y\ar[r]&\Om^\bullet_V\ar[r]& \ul\Om^\bullet_E,
\end{tikzcd}
\begin{tikzcd}
H^2(Y,\C)\ar[r,"\al"]\ar[d]&H^2(V,\C)\ar[r,equal]\ar[d]& H^2(E,\C)\ar[d] \\
H^2(Y,\O_Y)\ar[r,"\be"]&H^2(V,\O_V)\ar[r]& H^2(E,\O_E).
\end{tikzcd}\end{equation}
Now $H^2(Y,\C)$ has a pure Hodge structure and $H^2(E,\C)$ a mixed Hodge structure. The composite map $H^2(Y,\C)\to H^2(E,\C)$ is a morphism of mixed Hodge structures and there is an induced map $H^2(Y,\O_Y)=\gr^0_FH^2(Y,\C)\to \gr^0_FH^2(E,\C)=H^2(E,\O_E).$ The latter agrees with the composite of the bottom two arrows in \eq{YVE}. The left and right vertical arrows in \eq{YVE} are the edge homomorphisms corresponding to the resolutions of $\C_Y,\C_E$ respectively. As in \cite[Lemma-Definition 3.4]{PSt} there are splittings of
$H^2(Y,\C),H^2(E,\C)$ from which we get sections of the two vertical arrows in \eq{YVE} which commute with the horizontal arrows. This implies that the map $\ker\al\to \ker\be$ has a section, which we denote by $s:\ker\be\to\ker\al.$ Notice that the commutative diagram \eq{YVE} makes sense with $\R$ in place of $\C.$ The corresponding map $H^2(Y,\R)\to H^2(V,\R)=H^2(E,\R)$ is denoted by $\al'.$ The complex conjugate $\bar s$ is well defined and the map $s+\bar s:\ker\be\to\ker\al'$ is a section of the map $\ker\al'\to \ker \be.$ 

If $\cF$ is a sheaf on $Y$ then its Leray spectral sequence relative to $\pi:Y\to X$ is denoted by $E_2^{pq}(\cF)=H^p(X,R^p\pi_*\cF)\Rightarrow H^2(Y,\cF).$ Take $\cF$ to be the constant sheaf $\R_Y$ or the structure sheaf $\O_Y.$ In both cases $R^q\pi_*\cF,q>0,$ is supported on the finite set $X^\sing;$ and in particular, we have $E_2^{11}(\cF)=0.$ There is accordingly an exact sequence $E_2^{20}(\cF)\to H^2(Y,\cF)\to E_2^{02}(\cF)$ consisting of the edge homomorphisms. From the sheaf morphism $\R_Y\to\O_Y$ we get a commutative diagram
\[\begin{tikzcd}
H^2(X,\R)=E_2^{20}(\R_Y)\ar[r]\ar[d,"\ga"]& H^2(Y,\R)\ar[r,"\al' "]\ar[d] &E_2^{02}(\R_Y)=H^2(V,\R)\ar[d]\\
H^2(X,\O_X)=E_2^{20}(\O_Y)\ar[r]& H^2(Y,\O_Y)\ar[r,"\be"] &E_2^{02}(\O_Y)=H^2(V,\O_V)
\end{tikzcd}\]
with exact rows. Since the map $\ker\al'\to\ker\be$ has a section it follows by diagram chase that $\ga$ is surjective. 
\end{proof}
Combining Lemma \ref{lem: Kahler} and Theorem \ref{thm: Bing} we obtain
\begin{cor}\label{cor: Kahler}
Let $X$ be a compact normal complex space with isolated singularities, and $\cX\to S$ a deformation of $X.$ Then the set of $s\in S$ such that the fibre $X_s$ is K\"ahler is open with respect to the underlying topology of $S.$ 
\end{cor}

\section{Log-Canonical Singularities}\label{sect: lc}
We begin by recalling the definition of log-canonical singularities and proving a simple lemma about them. We suppose throughout that the complex spaces are normal and Gorenstein.
\begin{dfn}\l{dfn: lc}
Let $X$ be a Gorenstein normal complex space and $\om_X$ its canonical sheaf. We say that $X$ is {\it log-canonical} if the following holds: let $\pi:(Y,E)\to(X,X^\sing)$ be a good resolution with $E=\sum_iE_i,$ each $E_i$ irreducible, and define $a_i\in\Z$ by $\om_Y=\pi^*\om_X(\sum_i a_i E_i);$ then $a_i\ge-1$ for every $i.$ Each $a_i$ is called the {\it discrepancy coefficient of $E_i.$}  
\end{dfn}
\begin{lem}\label{lem: lc}
Let $X$ be a Gorenstein normal complex space. Then {\bf(i)} $X$ is log-canonical if and only if {\bf(ii)} for every good resolution $\pi:(Y,E)\to(X,X^\sing)$ we have $\pi_*\om_Y(E)=\om_X.$ 
\end{lem}
\begin{proof}
As the statement is local and as $X$ is Gorenstein, we may suppose that $\om_X\cong\O_X.$ Let (ii) hold and write $\om_Y=\pi^*\om_X(\sum_i a_iE_i)$ as in Definition \ref{dfn: lc}. Then
\[\ts H^0(X,\O_X)=H^0(X,\om_X)=H^0(Y,\om_Y(E))=H^0(Y,\O_Y(\sum_i(a_i+1)E_i)).\] 
But $1\in H^0(X,\O_X)$ and so $a_i+1\ge0$ for every $i;$ that is, (i) holds. Conversely, if $a_i\ge-1$ then $X$ normal implies $\O_X=\pi_*\O_Y(\sum_i(a_i+1)E_i).$ The latter sheaf is equal to $\pi_*\om_Y(E)$ as we have to prove.
\end{proof}
The following theorem is the main result of this section.
\begin{thm}\l{thm: lc}
Let $X$ be a complex space and $f:X\to S$ a flat morphism whose fibre over $s\in S$ is denoted by $X_s.$ Then the set of $x\in X$ such that $x\in X_{f(x)}$ is either an isolated log-canonical Gorenstein normal singularity or a non-singular point is an open subset of $X.$
\end{thm}
\begin{proof}
For $S=\C$ the theorem is a result of Ishii \cite[Proposition 4.4]{Ish small} (and there is a more general result called the inversion of adjunction theorem \cite{Fuj} proved using the minimal model program). We deduce from this the general case. 

It is well known that the set of $x\in X_{f(x)}$ such that $x\in X_{f(x)}$ is either an isolated Gorenstein normal singularity or a non-singular point is an open subset of $X.$ Suppose therefore that for every $s\in S$ the fibre $X_s$ has isolated Gorenstein normal singularities and for some $o\in S$ the fibre $X_o$ has a unique log-canonical singulartiy. We prove then that for $s$ close enough to $o$ the singularities of $X_s$ are log-canonical.

Suppose first that $S$ is reduced. By Theorem \ref{thm: strat} there exists a stratification $S=S^0\sqcup\dots\sqcup S^m.$ For $j=0,\dots,m$ put $X^j:=S^j\times_SX.$ Since $S^0,\dots,S^m$ are non-singular it follows that $X^0,\dots,X^m$ are Gorenstein normal complex spaces. Fix $i\in\{0,\dots,m\}$ for the moment. Denote by $\Si\sb S^j$ the set of $s\in S^j$ such that $X_s$ is not log-canonical. We show that it is an analytic subset of $S^j.$ By the property of $S^j$ there exists a good resolution $g:(Y,E)\to (X^j,(X^j)^\sing)$ of the projection $X^j\to S^j.$ The inclusion $g_*\om_Y(E)\to\om_{X^j}$ is denoted by $\io.$ For $s\in S^j$ the discrepancy coefficients of $X_s$ are locally constant functions of $s.$ So a point $x\in X_s$ is log-canonical if and only if $x\in X^j$ is log-canonical. By Lemma \ref{lem: lc} this holds if and only if the stalk morphism $\io_x$ is an isomorphism. Thus $\Si=f(\supp(\ker\io\op\coker\io)).$ Since $f|_{X^\sing}:X^\sing\to S$ is proper it follows that $\Si\sb S^j$ is an analytic subset. 

Denote by $\ov\Si\sb S$ the Zariski closure of $\Si\sb S.$ Suppose $o\in\ov\Si.$ Taking a resolution of $\ov\Si\-S^j\subset \ov\Si$ we see that the curve selection lemma holds; that is, there exist an open neighbourhood $\De$ of $0\in\C$ and a morphism $(\De,0)\to(S,o)$ which maps $\De\-\{0\}$ to $\ov\Si\cap S^j=\Si.$ But $X_o$ is log-canonical, so by \cite[Proposition 4.4]{Ish small} there exists $s\in\Si$ such that $X_s$ is log-canonical, which contradicts the definition of $\Si.$ Thus $o$ does not lie in $\ov\Si.$ 

Replacing $S$ by $S\-\ov\Si$ it follows that if $s\in S^j$ then $X_s$ is log-canonical. Repeating this process for $i=0,\dots,m$ it follows that every $X_s$ is log-canonical. If $S$ is not reduced, take its irreducible components and shrink them as has just been done, which completes the proof. 
\end{proof}

\section{Du Bois Singularities}\label{sect: DB}
For $p,q=0,1,2,\dots$ the {\it Du Bois invariant} $b^{pq}(X,x)$ and the {\it link invariant} $l^{pq}(X,x)$ are defined to be the respective dimensions of
\begin{equation}\label{DB link}
H^q(Y,\Om^p_Y(\log E)(-E)) \text{ and } H^q(E,\Om^p_Y(\log E)\otimes_{\O_Y}\O_E).
\end{equation}
The two invariants are independent of the choice of $f.$ After replacing $X$ by a contractible Stein neighbourhood of $x$ the local duality holds; that is, 
\[\text{$H^q(Y,\Om^p_Y(\log E)(-E))$ is dual to $H^{n-q}_E(Y,\Om^{n-p}_Y(\log E))$}\]
and $l^{pq}(X,x)=l^{n-p,n-q-1}(X,x).$ By the results of \cite{GNPP,St85}, if $p+q>n$ then $b^{pq}(X,x)=0.$

The notion of Du Bois singularities was introduced in \cite{St83} using the results of \cite{DB}; see also \cite[Definition 7.34]{PSt}. For $X$ normal, the isolated singularity $x\in X$ is Du Bois if and only if for $q=1,2,3,\dots$ we have $b^{0q}(X,x)=0.$ By \cite[Theorem 1 and Proposition 1]{St97}, if $X$ is Cohen--Macaulay of dimension $n$ then $b^{0q}(X,x)=0$ unless $q=n-1.$ Du Bois singularities and log-canonical singularities are related as follows.
\begin{thm}[Ishii {\cite[Theorem 2.3]{Ish G}}]\label{thm: iso Goren DB}
Let $X$ be a Gorenstein normal complex space and $x\in X$ an isolated singularity. Then $x$ is log-canonical if and only if $x$ is Du Bois. \qed
\end{thm}
\begin{rmk}
For algebraic varieties there is a more general result \cite{KSS}: a Cohen--Macaulay normal algebraic variety $X$ (which may have non-isolated singularities) is Du Bois if and only if it satisfies the condition (ii) of Lemma \ref{lem: lc}.
\end{rmk}

Here is a base change theorem for Du Bois singularities taken from \cite[Theorem 1.1 and Remark 4.3]{FL2}; the original versions are in \cite{DB,DBJ}.
\begin{thm}\label{thm: DB}
Let $f:X\to S$ be a proper flat morphism of complex spaces and $s\in S$ a point over which the fibre $X_s$ has isolated Du Bois singularities and its resolutions satisfy the the $\partial\db$ lemma. Then for $q=0,1,2,\dots$ the $\O_{S,s}$ module $(R^qf_*\O_X)_s$ is free and compatible with base change. \qed
\end{thm}

\begin{cor}\label{cor: DB}
Let $f:X\to S$ be a proper flat morphism of complex spaces and $s\in S$ a point over which the fibre $X_s$ is a compact log-canonical Gorenstein complex space with trivial canonical sheaf, with isolated singularities and whose resolutions satisfy the $\partial\db$ lemma. Then after replacing $S$ by a neighbourhood of $s$ the relative canonical sheaf $\om_{X/S}$ is a rank-one free $\O_X$ module.
\end{cor}
\begin{proof}
Let $X_s$ have pure dimension $n.$ By Theorem \ref{thm: DB} the map $(R^nf_*\O_{X/S})_s\otimes_{\O_{S,s}}\C\to H^n(X_s,\O_{X_s})$ is an isomorphism. By relative duality this induces an isomorphism $(R^0f_*\om_{X/S})_s\otimes_{\O_{S,s}}\C\to H^0(X_s,\om_{X_s}).$ Since $\om_{X_s}\cong\O_{X_s}$ we get an $\O_{S,s}$ module isomorphism $(R^0f_*\om_{\cX/S})_s\cong\O_{S,s}.$ The generator of $(R^0f_*\om_{X/S})_s$ defines near $X$ a nowhere-vanishing section of $\om_{X/S}.$ 
\end{proof}

\section{Vanishing Theorems}\label{sect: vanishing}
We begin with the following vanishing theorem.
\begin{thm}[Steenbrink {\cite[\S2]{St97}}]\label{thm: DB van}
Let $X$ be a complex space of dimension $n,$ and $x\in X$ an isolated Du Bois singularity. Then $b^{1,n-1}(X,x)=0.$ \qed
\end{thm}
\begin{cor}\label{cor: 1n-2}
Let $X$ be a complex space of dimension $n$ and $x\in X$ an isolated Du Bois singularity. Then
\[\dim_\C H^2_E(Y,\Om^{n-1}_Y(\log E)(-E))=l^{1,n-2}(X,x)+b^{1,n-2}(X,x).\]
\end{cor}
\begin{proof}
The local cohomology sequence induces an exact sequence
\[0\to H^1_E(\ts\frac{\Om^{n-1}_Y(\log E)}{\Om^{n-1}_Y(\log E)(-E)})\to H^2_E(\Om^{n-1}_Y(\log E)(-E))\to H^2_E(\Om^{n-1}_Y(\log E))\to 0.\]
But $H^1_E(\ts\frac{\Om^{n-1}_Y(\log E)}{\Om^{n-1}_Y(\log E)(-E)})=H^1(E,\Om^{n-1}_Y(\log E)\otimes_{\O_Y}\O_E)$ whose dimension is by duality equal to $l^{1,n-2}(X,x).$ Also $H^2_E(\Om^{n-1}_Y(\log E))$ is dual to $H^{n-2}(\Om^1_Y(\log E)(-E))$ which has dimension $b^{1,n-2}(X,x).$ This completes the proof.
\end{proof}

We turn now to the vanishing theorem of Popa, Shen and Vo \cite[Proposition 5.2]{PSV}. 
\begin{thm}\label{thm: PSV}
Let $X$ be a compact complex space with isolated Du Bois singularities and whose resolutions satisfy the $\partial\db$ lemma. Let $\pi:(Y,E)\to(X,X^\sing)$ be a good resolution. Put $\cF:=\Omega^1_Y(\log E)(-E)$ and denote by $E_2^{pq}=H^p(X,R^q\pi_*\cF)\Rightarrow H^{p+q}(Y,\cF)$ the Leray spectral sequence. Then for every $p=0,1,2,\dots$ the edge homomorphism $H^p(Y,\cF)\to E_2^{0p}$ vanishes.
\end{thm}
\begin{proof}
Since $R^q\pi_*\cF, q>0,$ is supported on the finite set $X^\sing$ it follows that for $p,q>0$ we have $E_2^{pq}=0.$ So there is an exact sequence $E_2^{p0}\to H^p(Y,\cF)\to E_2^{0p}$ consisting of the edge homomorphisms. But by \cite[Proposition 5.2]{PSV} the map $E_2^{p0}\to H^p(Y,\cF)$ is surjective; and consequently, the map $H^p(Y,\cF)\to E_2^{0p}$ vanishes.
\end{proof}
\begin{rmk}
We shall in practice take $p=n-2,$ in which case the result is reproved in \cite[Remark 2.5]{Friedman}.
\end{rmk}
\begin{cor}[Proposition 4.5 of \cite{Ten}]\label{cor: surj log}
Let $X$ be a compact log-canonical Gorenstein complex space of pure dimension $n,$ with isolated singularities and whose resolutions satisfy the $\partial\db$ lemma. Let $\pi:(Y,E)\to(X,X^\sing)$ be a good resolution. Then the natural map $H^{q-1}(X^\reg,\Om^{n-1}_{X^\reg}(\log E))\to H^q_E(Y,\Om^{n-1}_Y(\log E))$ is surjective.
\end{cor}
\begin{proof}
Take $q=n-p.$ The dual map $(E_2^{0p})^*\to H^p(Y,\cF)^*$ vanishes; that is, the map $H^q_E(Y,\Om^{n-1}_Y(\log E))\to H^q(Y,\Om^{n-1}_Y(\log E))$ vanishes. The local cohomology exact sequence implies therefore the conclusion above. 
\end{proof}
\begin{rmk}
In Corollary \ref{cor: surj log} we do not need the singularities to be complete intersections.
\end{rmk}

\section{Proof of Theorem \ref{thm: CY1}}\label{sect: proof of CY1}
Combining Theorem \ref{thm: surj tang} and Corollary \ref{cor: surj log} we prove
\begin{lem}\label{lem: CY1}
Let $X$ be a compact log-canonical Gorenstein complex space of dimension $n\ge3,$ with trivial canonical sheaf, with isolated singularities and whose resolutions satisfy the $\partial\db$ lemma. Suppose that there exists a good resolution $(\cY,\cE)\to(\cX,\cX^\sing)$ of the morphism $\cX\to\Def(X)$ which is a semi-universal deformation of $X.$ Then the singularities of $X$ have $b^{1,n-2}=0.$
\end{lem}
\begin{proof}
Denote by $o\in \Def(X)$ the origin and by $(Y,E)\sb (\cY,\cE)$ the fibre over it of the composite morphism $\cY\to\Def(X).$ The induced morphism $(Y,E)\to(X,X^\sing)$ is a good resolution. Since the canonical sheaf of $X$ is trivial, we get on $X^\reg$ a nowhere-vanishing holomorphic $n$-form. As $X$ is log-canonical, the pull-back of the $n$-form defines a section of $\om_Y(E).$ Contraction with it defines a sheaf morphism $\Th_Y(-\log E)\to \Om^{n-1}_Y(\log E).$ The image of this is generically a holomorphic $n-1$ form; more precisely, the results of the contraction are, at every point of $E$ lying on exactly one irreducible component of it, local sections of $\Om^{n-1}_Y.$ But $\Om^{n-1}_Y$ is a locally free sheaf on $Y$ and every non-trivial intersection of the irreducible components of $E$ has co-dimension $\ge2.$ So the morphism $\Th_Y(-\log E)\to\Om^{n-1}_Y(\log E)$ has image in $\Om^{n-1}_Y.$ 

Let $(\Def(Y,E),o)$ be a Kuranishi space germ of $(Y,E).$ By Theorem \ref{thm: surj tang} the map $T_o\Def(Y,E)\to T_o\Def(X)$ is surjective. We have $T_o\Def(Y,E)\cong H^1(Y,\Th_Y(-\log E))$ and $T_o\Def(X)\cong \Ext^1(\Om_X,\O_X).$ By Schlessinger's theorem we have $\Ext^1(\Om_X,\O_X)\cong H^1(X^\reg,\Th_{X^\reg})\cong H^1(X^\reg,\Om^{n-1}_{X^\reg}).$ So the map $H^1(Y,\Th_Y(-\log E))\to H^1(X^\reg,\Om^{n-1}_{X^\reg})$ is surjective. By the local cohomology exact sequence the map $H^1(X,\Om^{n-1}_{X^\reg})\to H^2_E(Y,\Th_Y(-\log E))$ vanishes. Using the sheaf morphisms $\Th_Y(-\log E)\to \Om^{n-1}_Y\to \Om^{n-1}_Y(\log E)$ we see that 
\begin{equation}\label{conn van}\parbox{10cm}{
the connecting homomorphisms $H^1(X,\Om^{n-1}_{X^\reg})\to H^2_E(Y,\Om^{n-1}_Y)$ and $H^1(X,\Om^{n-1}_{X^\reg})\to H^2_E(Y,\Om^{n-1}_Y(\log E))$ vanish. 
}\end{equation}
By Corollary \ref{cor: surj log} with $q=2$ the latter map is surjective. So $H^2_E(Y,\Om^{n-1}_Y(\log E))$ vanishes; that is, the singularities of $X$ have $b^{1,n-2}=0.$
\end{proof}

We now restate and prove Theorem \ref{thm: CY1}. 
\begin{thm}\label{thm: CY1 again}
Let $X$ be a compact log-canonical Gorenstein K\"ahler space of dimension $n\ge3,$ with isolated singularities and such that $\om_X\cong\O_X.$ Then there exist a semi-universal deformation $\cX\to \Def(X)$ and a Zariski dense open subset $\De\sb \Def(X)$ such that for every $t\in\De$ the singularities of $\cX_t$ have $b^{1,n-2}=0.$
\end{thm} 
\begin{proof}
Suppose first that $\Def(X)$ is irreducible. For $x\in X^\sing $ the restriction morphism $\Def(X)\to\Def(X,x)$ is denoted by $\ph_x.$ By Theorem \ref{thm: strat} there exists a stratification $\Def(X,x)=S_0\sqcup\dots\sqcup S_m.$ The least integer $l$ for which $f_x(\Def(X))\sb S_0\sqcup\dots\sqcup S_l$ as germs is denote by $l(x).$ Put $A_x:=S_0\sqcup\dots\sqcup S_{l(x)-1}$ (which is empty if $l(x)=0$). It is an analytic subset and so is $\ph_x^{-1}(A_x)\sb\Def(X).$ By the definition of $l(x)$ we have $\ph_x^{-1}(A_x)\ne\Def(X)$ as germs. But $\Def(X)$ is irreducible and so $\ph_x^{-1}(A_x)$ has empty interior. In other words, $\ph_x^{-1}(S_{l(x)})\sb\Def(X)$ is a Zariski dense open subset. Set $\De:=\bigcap_{x\in X^\sing}\ph_x^{-1}(S_{l(x)}).$ Then $\De\sb\Def(X)$ is a Zariski dense open subset. 

If $\Def(X)$ has several irreducible components, we get a Zariski dense open subset of each irreducible component. Define $\De$ to be their union. Then $\De\sb\Def(X)$ is a Zariski dense open subset. 

Let $\cU\to\Def(X,X^\sing)$ be a semi-universal deformation of the germ $(X,X^\sing)$ such that for every $s\in \Def(X,X^\sing)$ the flat morphism $\cU\to\Def(X,X^\sing)$ is a versal deformation of the germ $(\cU_s,\cU_s^\sing).$ Make $\Def(X)$ so small that the restriction morphism $\Def(X)\to\Def(X,X^\sing)$ is well defined. We denote it by $\ph.$ 

Corollary \ref{cor: Kahler}, Theorem \ref{thm: lc} and Corollary \ref{cor: DB} imply that after making $\Def(X)$ smaller if necessary, there exists a semi-universal deformation $\cX\to\Def(X)$ whose fibres $\cX_t, t\in\Def(X),$ are log-canonical Gorenstein K\"ahler spaces with isolated singularities and with trivial canonical sheaf. 

We show that for every $t\in \De\sb\Def(X)$ the singularities of $\cX_t$ have $b^{1,n-2}=0.$ Let $\cY\to\Def(\cX_t)$ be a semi-universal deformation of $\cX_t.$ Since $\cU\to\Def(X,X^\sing)$ is a versal deformation of $(\cX_t,\cX_t^\sing)$ we get, after making $\Def(\cX_t)$ smaller if necessary, a morphism $\Def(\cX_t)\to\Def(X,X^\sing)$ under which the origin of $\Def(\cX_t)$ maps to $\ph(t).$ As $\ph(t)\in \prod_{x\in X^\sing}S_{l(x)}$ there exists a good resolution of the morphism $\Def(\cX_t)\times_{\Def(X,X^\sing)}\cU\to \Def(\cX_t).$ It extends to a good resolution of $\cY\to\Def(\cX_t).$ By Lemma \ref{lem: CY1} the singularities of $\cX_t$ have $b^{1,n-2}=0.$
\end{proof}

\section{Proof of Theorem \ref{thm: CY2}}\label{sect: proof of CY2}
We recall first the Goresky--MacPherson theorem.  
\begin{thm}[Corollary 1.12 of \cite{St83}]\label{thm: GM}
Let $X$ be a complex $n$-fold, $x\in X$ an isolated singularity and $f:(Y,E)\to(X,x)$ a good resolution. Then for integers $p<n$ and $q>n$ there are exact sequences
\[\begin{split}
0\to H^p(Y,Y\-E;\C)\to H^p(E,\C)\to H^{p+1}(X,X\-\{x\};\C)\to 0,\\ 0\to H^q(X,X\-\{x\};\C)\to H^q(Y,Y\-E;\C)\to H^q(E,\C)\to 0
\end{split}\]
consisting of the morphisms of mixed Hodge structures. Also the map $H^n(Y,Y\-E;\C)\to H^n(E,\C)$ is an isomorphism of mixed Hodge structures. \qed
\end{thm}

Using the techniques of Friedman--Laza \cite[Theorem 2.1(iii),(v)]{FL} we prove
\begin{lem}\label{lem: inj}
Let $X$ be a complex space of dimension $n,$ $x\in X$ an isolated Du Bois singularity and $(Y,E)\to(X,x)$ be a good resolution. Then the map
\begin{equation}\label{inj}
H^2_E(Y,\Om^{n-1}_Y(\log E)(-E))\to H^2_E(Y,\Om^{n-1}_Y)
\end{equation}
induced by the inclusion $\Om^{n-1}_Y(\log E)(-E)\to \Om^{n-1}_Y$ is injective. 
\end{lem}
\begin{proof}
There is a commutative diagram
\[\begin{tikzcd}
0\ar[r]&\Om^{n-1}_Y\ar[r]\ar[d]&\Om^{n-1}(\log E)\ar[r]\ar[d]& \frac{\Om^{n-1}_Y(\log E)}{\Om^{n-1}_Y}\ar[r]\ar[d,equal]&0\\
0\ar[r]&\frac{\Om^{n-1}_Y}{\Om^{n-1}_Y(\log E)(-E)}\ar[r]&\frac{\Om^{n-1}_Y(\log E)}{\Om^{n-1}_Y(\log E)(-E)}\ar[r]& \frac{\Om^{n-1}_Y(\log E)}{\Om^{n-1}_Y}\ar[r]&0
\end{tikzcd}\]
with exact rows. The corresponding diagram
\begin{equation}\label{FL2}\begin{tikzcd}
H^0_E(Y,\frac{\Om^{n-1}_Y(\log E)}{\Om^{n-1}_Y})\ar[r,"\al"] \ar[d,equal]& H^1_E(Y,\Om^{n-1}_Y)\ar[d,"\ga"]\\
H^0_E(Y,\frac{\Om^{n-1}_Y(\log E)}{\Om^{n-1}_Y})\ar[r,"\be"] & H^1_E(Y,\frac{\Om^{n-1}_Y}{\Om^{n-1}_Y(\log E)(-E)})
\end{tikzcd}\end{equation}
commutes. The map $\al$ is part of the exact sequence
\[0=H^0_E(Y,\Om^{n-1}_Y(\log E))\to H^0_E(Y,\ts\frac{\Om^{n-1}_Y(\log E)}{\Om^{n-1}_Y})\xrightarrow{\al} H^1_E(Y,\Om^{n-1}_Y)\to H^1_E(Y,\Om^{n-1}_Y).\]
By Theorem \ref{thm: DB van} the last term vanishes and $\al$ is therefore an isomorphism. 

The map $\be$ is induced from the morphism $H^n(Y,Y\-E;\C)\to H^n(E,\C)$ of mixed Hodge structures which is by Theorem \ref{thm: GM} an isomorphism. It is a strict morphism relative to the Hodge filtrations and $\be$ is therefore an isomorphism.

As $\al,\be$ are isomorphisms in \eq{FL2} the right vertical arrow $\ga$ is also an isomorphism. There is now a commutative diagram
\[\begin{tikzcd}[column sep=small]
 & H^1_E(Y,\Om^{n-1}_Y)\ar[r,"\ga"]\ar[d]& H^1_E(Y,\frac{\Om^{n-1}_Y}{\Om^{n-1}_Y(\log E)(-E)})\ar[d,equal]\\
H^1(Y,\Om^{n-1}(\log E)(-E))\ar[r]\ar[d,"\ze"]& H^1(Y,\Om^{n-1}_Y)\ar[d,"\ep"]\ar[r,"\de"]& H^1(Y,\frac{\Om^{n-1}_Y}{\Om^{n-1}_Y(\log E)(-E)})\\
H^1(Y\-E,\Om^{n-1}_{Y\-E})\ar[r,equal]\ar[d]& H^1(Y\-E,\Om^{n-1}_{Y\-E})\ar[d]&\\
H^2_E(Y,\Om^{n-1}_Y(\log E)(-E))\ar[r,"\eqref{inj}"]\ar[d]&H^2_E(Y,\Om^{n-1}_Y)&\\
H^2(Y,\Om^{n-1}_Y(\log E)(-E))=0&&
\end{tikzcd}\]
where the second row, the left column and the middle column are exact. As $\ga$ is an isomorphism, the map $\de$ above has a splitting. This implies that the image of $\ep$ is equal to that of $\ze.$ Diagram chase shows then that \eqref{inj} is injective.
\end{proof}

The following lemma is the key to the proof of Theorem \ref{thm: CY2}.
\begin{lem}\label{lem: conn2}
Let $X$ be a compact complex $n$-fold with isolated Du Bois singularities, whose resolutions satisfy the $\partial\db$ lemma and such that $H^n(X^\reg,\C)\to H^{n+1}(X,X^\reg;\C)$ is surjective. Then for every good resolution  $(Y,E)\to (X,X^\sing)$ the connecting homomorphism
\e\l{conn} 
H^1(X^\reg,\Om^{n-1}_{X^\reg})\to H^2_E(Y,\Om^{n-1}_Y(\log E)(-E))
\e
is surjective.
\end{lem}
\begin{proof}
We begin with a general definition. Let $\cF$ be a locally free $\O_Y$ module. In the derived category of $\O_Y$ modules we define a commutative diagram
\begin{equation}\label{81}
\begin{tikzcd}
\cF(-E)\ar[r]\ar[d,equal]&\cF\ar[r]\ar[d] &\cF\otimes_{\O_Y}\O_E\ar[d]\\
\cF(-E)\ar[r]\ar[d]&\cF_{Y\-E}\ar[r]\ar[d,equal]& R\Ga_E\cF(-E)[1]\ar[d]\\
\cF \ar[r]& \cF_{Y\-E}\ar[r,equal]&R\Ga_E\cF[1].
\end{tikzcd}
\end{equation}
whose rows are distinguished triangles. The top row is the distinguished triangle corresponding to the sheaf exact sequence $0\to \cF(-E)\to\cF\to\cF\otimes_{\O_Y}\O_E\to0.$ The two other rows are obtained from the distinguished triangle $R\Ga_E\cG\to \cG\to \cG|_{Y\-E}\to R\Ga_E\cF[1]$ where $\cG$ is a sheaf on $Y.$ Taking $\cG=\cF(-E)$ and making the identification $\cF(-E)|_{Y\-E}=\cF|_{Y\-E}$ we get the middle row. Taking $\cG=\cF$ we get the bottom row. In the left column the morphism $\cF(-E)\to\cF$ is that of the top row. In the middle column the morphism $\cF\to\cF|_{Y\-E}$ is that of the bottom row. The right column is the distinguished triangle obtained from the octahedron axiom (which is therefore not unique).

Applying \eq{81} to $\cF=\Om^{n-1}_Y(\log E)$ and taking the long exact sequences we get a commutative diagram
\begin{equation}\label{oct}\begin{tikzcd}
H^1(Y,\Om_Y^{n-1}(\log E))\ar[r]\ar[d] & H^1(Y,\Om_Y^{n-1}(\log E)\otimes_{\O_Y}\O_E)\ar[d]\\
H^1(X^\reg,\Om^{n-1}_{X^\reg})\ar[r]\ar[d,equal] & H^2_E(Y,\Om_Y^{n-1}(\log E)(-E)) \ar[d]\\
H^1(X^\reg,\Om^{n-1}_{X^\reg})\ar[r] & H^2_E(Y,\Om_Y^{n-1}(\log E)).
\end{tikzcd}\end{equation}
The top horizontal arrow is induced from the morphism $H^n(X^\reg,\C)\to H^n(L,\C)$ of mixed Hodge structures; see for instance \cite[Theorem 4.2 and Corollary 6.14]{PSt}. By hypothesis the map $H^n(X^\reg,\C)\to H^n(L,\C)$ is surjective. As it is a strict morphism relative to the Hodge filtrations, the top horizontal arrow of \eq{oct} is surjective. 

The two other horizontal arrows are the connecting homomorphisms. By Corollary \ref{cor: surj log} the bottom one is surjective. Diagram chase shows that the middle one is surjective, completing the proof.
\end{proof}

Combining Lemmas \ref{lem: inj} and \ref{lem: conn2} we prove
\begin{cor}\l{cor: key surj2}
Let $X$ be as in Lemma \ref{lem: conn2} with $n\ge3.$ Suppose that there exists a good resolution of the morphism $\cX\to\Def(X)$ which is a semi-universal deformation of $X.$ Then for every $x\in X$ we have $b^{1,n-2}(X,x)=l^{1,n-2}(X,x)=0.$
\end{cor}
\begin{proof}
By hypothesis \eq{conn van} holds; that is, the composite map
\[H^1(X^\reg,\Om^{n-1}_{X^\reg})\xrightarrow{\eqref{conn}} H^2_E(Y,\Om^{n-1}_Y(\log E)(-E))\xrightarrow{\eqref{inj}} H^2_E(Y,\Om^{n-1}_Y)\]
vanishes. By Lemma \ref{lem: inj} the latter map \eqref{inj} is injective and so \eqref{conn} vanishes. By Lemma \ref{lem: conn2} however it is surjective. Consequently, $H^2_E(Y,\Om^{n-1}_Y(\log E)(-E))=0.$ By Corollary \ref{cor: 1n-2} we have $b^{1,n-2}(X,x)=l^{1,n-2}(X,x)=0, x\in X.$
\end{proof}

We now restate and prove Theorem \ref{thm: CY2}. 
\begin{thm}\label{thm: CY2 again}
Let $X$ be as in Lemma \ref{lem: conn2} with $n\ge3$ and which is K\"ahler. Then there exist a semi-universal deformation $\cX\to \Def(X)$ and a Zariski dense open subset $S\sb\Def(X)$ such that for every $s\in S$ the singularities of $\cX_s$ have $b^{1,n-2}=l^{1,n-2}=0.$
\end{thm}
\begin{proof}
Let $\De\sb\Def(X)$ and $\ph:\Def(X)\to \Def(X,X^\sing)$ be as in the proof of Theorem \ref{thm: CY1 again}. We show that for every $t\in \De$ the singularities of $\cX_t$ have $b^{1,n-2}=l^{1,n-2}=0.$ This is the same as the last step to the proof of Theorem \ref{thm: CY1 again}, with 
Corollary \ref{cor: key surj2} in place of Lemma \ref{lem: CY1}.
\end{proof}

\section{Proof of Theorem \ref{4}}\label{sect: proof of 4}
We begin by recalling the relevant vanishing results for isolated rational singularities. 
\begin{lem}[Lemma 2 of \cite{St97}]\label{lem: l02}
Let $X$ be a normal complex space and $x\in X$ an isolated rational singularity. Then for $p=0,1,2,\dots$ we have $l^{p0}(X,x)=l^{0p}(X,x)=0.$ \qed
\end{lem}
\begin{cor}\label{cor: l02}
Let $X$ be a normal complex $3$-fold and $x\in X$ an isolated rational singularity with link $L.$ Then $H^3(L,\C)=\gr_F^2H^3(L,\C).$
\end{cor}
\begin{proof}
By duality $l^{12}(X,x)=l^{20}(X,x).$ By Lemma \ref{lem: l02} the latter vanishes, and so $l^{12}(X,x)=0.$ Again by Lemma \ref{lem: l02} we have $l^{30}(X,x)=l^{03}(X,x)=0.$ Thus $\gr_F^pH^3(L,\C)$ vanishes unless $p=2.$
\end{proof}

\begin{lem}\label{lem: H4}
Let $X$ be a normal complex $3$-fold and $x\in X$ an isolated rational singularity with link $L.$ Then $H^4(L,\C)=0.$
\end{lem}
\begin{proof}
By \cite[Lemma 12.1.1.1]{MK} we have $H^1(E,\C)=0.$ By Theorem \ref{thm: GM} this implies $H^1(L,\C)=0.$ By Poincar\'e duality therefore $H^4(L,\C)=0.$
\end{proof}

We prove now the former part of Theorem \ref{4}. Let $X$ be as in Theorem \ref{4}, that is, a compact normal complex $3$-fold with isolated rational singularities and whose resolutions satisfy the $\partial\db$ lemma. Denote by $L=\coprod_{x\in X^\sing}L_x$ the union of the links. Let $(Y,E)\to(X,X^\sing)$ be a good resolution. Then \eq{top} is necessary and sufficient for the restriction homomorphism $H^3(X^\reg,\C)\to H^3(L,\C)$ to be surjective. The inequality \eq{DB complex} is equivalent to 
\begin{equation}\label{h0}
\dim_\C H^2(Y,\Om^2_Y(\log E)(-E))\le\dim_\C H^2(Y,\Om^2_Y(\log E)).
\end{equation}
\begin{prop}\label{prop: 41}
\eq{h0} is necessary and sufficient for $H^3(X^\reg,\C)\to H^3(L,\C)$ to be surjective.
\end{prop}
\begin{proof}
By Corollary \ref{cor: l02} and Lemma \ref{lem: H4} there is an exact sequence
\begin{equation}\label{h3}\begin{split}
H^1(Y,\Om^2_Y(\log E))\to H^3(L,\C)\to H^2(Y,\Om^2_Y(\log E)(-E))\\
\to H^2(Y,\Om^2_Y(\log E))\to \gr_F^2H^4(L,\C)=0.\end{split} \end{equation}
If \eq{h0} holds then the map $H^2(Y,\Om^2_Y(\log E)(-E))\to H^2(Y,\Om^2_Y(\log E))$ is an isomorphism and so the map $H^1(Y,\Om^2_Y(\log E))\to H^3(L,\C)$ is surjective. Now the map $H^3(X^\reg,\C)\to H^3(L,\C)$ is a morphism of mixed Hodge structures and by Corollary \ref{cor: l02} the map $\gr^p_FH^3(X^\reg,\C)\to\gr_F^pH^3(L,\C)$ vanishes for $p\ne 2.$ So $H^3(X^\reg,\C)\to H^3(L,\C)$ is surjective.   

Conversely, if $H^3(X^\reg,\C)\to H^3(L,\C)$ is surjective then by \eq{h3} the map $H^2(Y,\Om^2_Y(\log E)(-E))\to H^2(Y,\Om^2_Y(\log E))$ is an isomorphism, which implies \eq{h0}.
\end{proof}
\begin{rmk}
The exact sequence \eq{h3} shows that \eq{h0} is necessary and sufficient for the equality to hold in \eq{h0}. In terms of the mixed Hodge structure, the latter is equivalent to the equality
\begin{equation}\label{h00}\dim_\C\gr_F^2H^4(X,\C)=\dim_\C\gr_F^1H^2(X,\C).\end{equation}
We show that \eq{h00} is equivalent to the Betti number equality $b^4(X)=b^2(X).$ By Poincar\'e duality $H^4(X^\reg,\C)$ is dual to the compact support cohomology group $H^2_c(X^\reg,\C)$ which is isomorphic to $H^2(X,X^\sing;\C).$ As $X^\sing$ is isolated, we have $H^2(X,X^\sing;\C)\cong H^2(X,\C).$ So there is an exact sequence
\begin{equation}\label{h4}
H^3(X^\reg,\C)\to H^3(L,\C)\to H^4(X,\C)\to H^2(X,\C)\to H^4(L,\C)=0
\end{equation}
where the last term vanishes by Lemma \ref{lem: H4}. As \eq{h4} is exact, $b^4(X)=b^2(X)$ if and only if $H^3(X^\reg,\C)\to H^3(L,\C)$ is surjective, which is equivalent to \eq{h00}.
\end{rmk}

We prove next the latter part of Theorem \ref{4}. 
\begin{prop}\label{prop: 43}
Let $X$ be an algebraic variety. Then $H^3(X^\reg,\C)\to H^3(L,\C)$ is surjective if and only if $X$ is $\Q$-factorial.
\end{prop}
\begin{proof}
The resolution $Y\to X$ induces the morphism
\begin{equation}\label{NS}\begin{tikzcd}
H^3(X^\reg,\C)\ar[r,"\al"]\ar[d,equal]& H^4(Y,Y\-E;\C)\ar[r]& H^4(Y,\C)\\
H^3(X^\reg,\C)\ar[r,"\be"]& H^4(L,\C)\ar[r]\ar[u]& H^4(X,\C)\ar[u]
\end{tikzcd}\end{equation}
between the relative cohomology exact sequences. Taking the dual of \eq{NS} we get a commutative diagram
\begin{equation}\label{NS3}\begin{tikzcd}
0&\im \tp\al\ar[d,equal]\ar[l]& H^2(E;\C)\ar[l]\ar[d]& H^2(Y,\C)\ar[d]\ar[l]\\
0&\im\tp\be\ar[l]& H^3(L,\C)^*\ar[l]& (\coker\be)^*\ar[l]&0.\ar[l]
\end{tikzcd}\end{equation}
Denote by $E=\sum_iE_i$ the irreducible decomposition. Then by Theorem \ref{thm: GM} we may write $H^3(L,\C)^*= H^2(E,\C)/\sum_i \C[E_i].$ By \eq{NS3} therefore
\begin{equation}\label{NS5}
(\coker\be)^*=\im(H^2(Y,\C)\to H^2(E,\C)/\ts\sum_i \C[E_i]).
\end{equation}
By Koll\'ar--Mori \cite[Proposition 12.1.6]{MK} the right-hand side of \eq{NS5} vanishes if and only if $X$ is $\Q$-factorial, which completes the proof.
\end{proof}
\begin{rmk}
We need $X$ algebraic at the last step. Otherwise, the proof is the same as that of Namikawa--Steenbrink \cite[Lemma 3.5]{NS}.
\end{rmk}

Theorem \ref{4} follows from Propositions \ref{prop: 41} and \ref{prop: 43}. \qed

We finally prove the statement in Remark \ref{rmk: Ten2}.
\begin{prop}\label{prop: 45}
Let $X$ be a compact complex $n$-fold with isolated complete intersection singularities and whose resolutions satisfy the $\partial\db$ lemma. Denote by $L=\coprod_{x\in X^\sing}L_x$ the union of the links and suppose that 
\begin{equation}\label{h5}
\dim_\C H^2(Y,\Om^{n-1}_Y(\log E)(-E))\le\dim_\C H^2(Y,\Om^{n-1}_Y(\log E))
\end{equation}
Then the equality holds in \eq{h5}.
\end{prop}
\begin{proof}
Consider the exact sequence
\begin{equation}\label{n+1 L}
H^2(Y,\Om^{n-1}_Y(\log E)(-E))\\
\to H^2(Y,\Om^{n-1}_Y(\log E))\to \gr_F^{n-1}H^{n+1}(L,\C).
\end{equation}
But as the singularities of $X$ are complete intersections, we have $H^{n+1}(L,\C)=0.$ The first arrow of \eq{n+1 L} is therefore surjective and accordingly by \eq{h5} an isomorphism. This completes the proof.
\end{proof}

\section{Proof of Theorem \ref{thm: Fano}}\label{sect: proof of Fano}

We modify Corollary \ref{cor: key surj2} as follows.
\begin{lem}\label{lem: Fano}
Let $X$ be a compact log-canonical Gorenstein complex space with isolated singularities and with $\om_X^{-1}$ ample. Suppose that there exists a good resolution of the morphism $\cX\to\Def(X)$ which is a semi-universal deformation of $X.$ Then for every $x\in X$ we have $b^{1,n-2}(X,x)=l^{1,n-2}(X,x)=0.$
\end{lem}
\begin{proof}
We use the contraction morphism $\Th_Y(-\log E)\to\Om^{n-1}_Y\otimes \pi^*\om_X^{-1}.$ The connecting homomorphism $H^1(X,\Th_{X^\reg})\to  H^2_E(Y,\Om^{n-1}_Y\otimes\pi^*\om_X^{-1})$ vanishes. Write it as the composite map
\e\l{Fan1}
H^1(X,\Th_{X^\reg})\to H^2_E(Y,\Om^{n-1}_Y(\log E)(-E)\otimes\pi^*\om_X^{-1})\to H^2_E(Y,\Om^{n-1}_Y\otimes\pi^*\om_X^{-1}).
\e
Choose trivializations of $\om_X$ at the singular points of $X.$ Lemma \ref{lem: inj} implies then that the latter map in \eq{Fan1} is injective. But the composite map vanishes, so the former map of \eq{Fan1} vanishes. By the results of \cite{GNPP,St85} we have $H^2(Y,\Om^{n-1}_Y(\log E)(-E)\otimes f^*\om_X^{-1})=0.$ So the former map of \eq{Fan1} is surjective, which implies that $H^2_E(Y,\Om^{n-1}_Y(\log E)(-E)\otimes\pi^*\om_X^{-1})=0.$ By Corollary \ref{cor: 1n-2} therefore $b^{1,n-2}(X,x)=l^{1,n-2}(X,x)=0, x\in X.$
\end{proof}
We restate and prove Theorem \ref{thm: Fano}.
\begin{thm}
Let $X$ be a projective log-canonical Gorenstein $n$-fold, with isolated singularities and such that $\om_X^{-1}$ is ample. Then there exist a semi-universal deformation $\cX\to \Def(X)$ and a Zariski dense open subset $\De\sb \Def(X)$ such that for every $t\in\De$ the singularities of $\cX_t$ have $b^{1,n-2}=l^{1,n-2}=0.$
\end{thm} 
\begin{proof}
Define $\De\sb\Def(X)$ in the same way as in the proof of Theorem \ref{thm: CY1 again}. Let $\cU\to\Def(X,X^\sing)$ and $\ph:\Def(X)\to \Def(X,X^\sing)$ be also as in the proof of Theorem \ref{thm: CY1 again}. 

Since $\om_X^{-1}$ is ample it follows by \cite[Theorem 1.2]{Fuj2} that $H^2(X,\O_X)=0.$ This and Theorem \ref{thm: lc} imply that after making $\Def(X)$ smaller if necessary, there exists a semi-universal deformation $\cX\to\Def(X)$ whose fibres $\cX_t, t\in\Def(X),$ are log-canonical Gorenstein K\"ahler spaces with isolated singularities and such that $\om_{\cX_t}^{-1}$ is ample. 

We show that for every $t\in \De\sb\Def(X)$ and every singular point $x\in\cX_t$ we have $b^{1,n-2}(\cX_t,x)=l^{1,n-2}(\cX_t,x)=0.$ This is the same as the last step to the proof of Theorem \ref{thm: CY1 again}, with Lemma \ref{lem: Fano} in place of Lemma \ref{lem: CY1}.
\end{proof}

Institute of Mathematical Sciences, ShanghaiTech University, 393 Middle Huaxia Road, Pudong New District, Shanghai, China 

e-mail address: yosukeimagi@shanghaitech.edu.cn

\end{document}